\documentclass[12pt]{article}%
\usepackage{amsfonts}
\usepackage{amsmath}
\usepackage{amssymb}
\usepackage{graphicx}
\usepackage{fancybox}
\usepackage{geometry}%
\newtheorem{theorem}{Theorem}
\newtheorem{acknowledgement}[theorem]{Acknowledgement}

\newtheorem{corollary}[theorem]{Corollary}

\newtheorem{definition}[theorem]{Definition}
\newtheorem{example}[theorem]{Example}

\newtheorem{lemma}[theorem]{Lemma}

\newtheorem{proposition}[theorem]{Proposition}
\newtheorem{remark}[theorem]{Remark}

\newenvironment{proof}[1][Proof]{\noindent\textbf{#1.} }{\ \rule{0.5em}{0.5em}}
\newlength{\querylen}
\begin{document}

\title{Self-dual continuous processes}
\date{}
\author{Thorsten Rheinl\"{a}nder\thanks{Department of Statistics, London School of
Economics, Houghton Street, London, WC2A 2AE, United Kingdom. phone +44 (0)20
7955 7169 fax +44 (0)20 7955 7416 (\texttt{T.Rheinlander@lse.ac.uk}).
Corresponding author}\ \thinspace\ and Michael Schmutz\thanks{Department of
Mathematical Statistics and Actuarial Science, University of Bern,
Sidlerstrasse 5, 3012 Bern, Switzerland. phone +41 (0)31 631 88 18 \ \ fax +41
(0)31 631 38 05 (\texttt{michael.schmutz@stat.unibe.ch})}}
\maketitle

\begin{abstract}
The important application of semi-static hedging in financial markets
naturally leads to the notion of quasi self-dual processes which is, for
continuous semimartingales, related to symmetry properties of both their
ordinary as well as their stochastic logarithms. We provide a structure result
for continuous quasi self-dual processes. Moreover, we give a characterisation
of continuous Ocone martingales via a strong version of self-duality.

\bigskip


\end{abstract}

\vspace{-5mm}

\medskip\noindent

\section{Introduction}

The duality principle in option pricing relates different financial products
by a certain change of measure. It allows to transform complicated financial
derivatives into simpler ones in a suitable dual market. For a comprehensive
treatment, see~\cite{eber:pap:shir08,eber:pap:shir08b} and the literature
cited therein.

Sometimes it is even possible to semi-statically hedge path-dependent barrier
options with European ones. These are options which only depend on the asset
price at maturity. Here semi-static refers to trading at most at inception and
a finite number of stopping times like hitting times of barriers. The
possibility of this hedge, however, requires a certain symmetry property of
the asset price which has to remain invariant under the duality
transformation, possibly after a power transform. This leads naturally to the
concepts of self-duality, resp. quasi self-duality, see~\cite{car:cho97} and
more recently~\cite{CL,MoS}. For references to the large literature of the
special case of put-call symmetry,
see~\cite{CL,eln:jem99,faj:mor06,faj:mor10,T}.

Continuous symmetric processes have been characterised in \cite{T}, and it is
shown therein that the conditional symmetry property is related to the
self-duality of their stochastic exponentials. We extend this study by
exploring the structure of quasi self-dual processes as well as characterising
continuous Ocone martingales using results from \cite{VY} and a strong version
of self-duality. Ocone martingales are a very important class of conditionally
symmetric martingales; indeed, Tehranchi raised in \cite{T} the question
whether \emph{all }conditionally symmetric martingales are Ocone. This
question is still open. We do provide, however, an example of a non-Ocone
martingale in continuous time which is process, but not conditionally symmetric.

\section{Definitions and general properties}

We work on a filtered probability space $\left(  \Omega,\mathcal{F}%
,\mathbb{F},P\right)  $ where unless otherwise stated, the filtration
satisfies the usual conditions with $\mathcal{F}_{0}$ being trivial up to
$P$-null sets, and fix a finite but arbitrary time horizon $T>0$. All
stochastic processes are RCLL and defined on $\left[  0,T\right]  $ unless
otherwise stated. We understand positive and negative in the strict sense.

\begin{definition}
\label{def:r-sym} Let $M$ be an adapted process. $M$ is \textbf{conditionally
symmetric} if for any stopping time $\tau\in\left[  0,T\right]  $ and any
non-negative Borel function $f$%
\begin{equation}
E\left[  \left.  f\left(  M_{T}-M_{\tau}\right)  \right\vert \mathcal{F}%
_{\tau}\right]  =E\left[  \left.  f\left(  M_{\tau}-M_{T}\right)  \right\vert
\mathcal{F}_{\tau}\right]  . \label{symmetry condition}%
\end{equation}

\end{definition}

Here it is permissible that both sides of the equation are infinite. If $M$ is
an integrable conditionally symmetric process, then
condition~(\ref{symmetry condition}) implies that $M$ is a martingale by
choosing $f(x)=x$ ($=x^{+}-x^{-}$).

\begin{definition}
Let $S$ be a positive adapted process. $S$ is \textbf{self-dual }if for any
stopping time $\tau\in\left[  0,T\right]  $ and any non-negative Borel
function $f$ we have%
\begin{equation}
E\left[  \left.  f\left(  \frac{S_{T}}{S_{\tau}}\right)  \right\vert
\mathcal{F}_{\tau}\right]  =E\left[  \left.  \frac{S_{T}}{S_{\tau}}f\left(
\frac{S_{\tau}}{S_{T}}\right)  \right\vert \mathcal{F}_{\tau}\right]  .
\label{self-duality}%
\end{equation}

\end{definition}

These definitions are new, and are motivated by the fact that in applications
to semi-static hedging one typically considers hitting times of barriers which
are stopping times. They differ from the ones used in~\cite{T} who uses
bounded measurable $f$ instead, and in particular deterministic times.
However, all corresponding results in~\cite{T} applied in this paper can be
adapted to our setting.

In the case when $S$ is a martingale, we can define a probability measure $Q$,
the so-called \emph{dual measure}, via%
\begin{equation}
\frac{dQ}{dP}=\frac{S_{T}}{S_{0}}. \label{measure Q}%
\end{equation}
Similarly, if $E\left[  \sqrt{S_{T}}\right]  <\infty$, or $E\left[  S_{T}%
^{w}\right]  <\infty$ for a $w\in\left[  0,1\right]  $, respectively, we
define probability measures $H$, sometimes called `half measure', respectively
$P^{w}$, via%
\begin{equation}
\frac{dH}{dP}=\frac{\sqrt{S_{T}}}{E\left[  \sqrt{S_{T}}\right]  }\,,\quad
\frac{dP^{w}}{dP}=\frac{S_{T}^{w}}{E\left[  S_{T}^{w}\right]  }\,.
\label{measure H}%
\end{equation}

Note that the integrability of $S_{T}=S_{0}\exp(X_{T})$ under $P$ implies the
existence of the moment generating function of $X_{T}$ under $H$ for an open
interval including the origin, i.e. $X_{T}$ has all moments under $H$.

By Bayes' formula, the self-duality condition~(\ref{self-duality}) can be
expressed for a martingale $S$ in terms of the dual measure $Q$ defined
in~(\ref{measure Q}) as%
\begin{equation}
E_{P}\left[  \left.  f\left(  \frac{S_{T}}{S_{\tau}}\right)  \right\vert
\mathcal{F}_{\tau}\right]  =E_{Q}\left[  \left.  f\left(  \frac{S_{\tau}%
}{S_{T}}\right)  \right\vert \mathcal{F}_{\tau}\right]  .
\label{self-duality numeraire}%
\end{equation}

\begin{lemma}
\textbf{(\cite{T}, Lemma 3.2.) }\label{Tehranchi lemma}\textbf{ }A positive
continuous martingale $S$ is self-dual if and only if%
\begin{equation}
E_{P}\left[  \left.  \left(  \frac{S_{T}}{S_{\tau}}\right)  ^{p}\right\vert
\mathcal{F}_{\tau}\right]  =E_{P}\left[  \left.  \left(  \frac{S_{T}}{S_{\tau
}}\right)  ^{1-p}\right\vert \mathcal{F}_{\tau}\right]
\label{self-duality lemma}%
\end{equation}
for all complex $p=a+ib$ with $a\in\left[  0,1\right]  $ and all stopping
times $\tau\in\left[  0,T\right]  $.
\end{lemma}

\medskip

For the measure $H$ (corresponding to $w=1/2$) the following proposition has
been stated in slightly different settings in~\cite{CL,MoS} and~\cite{T}, and
also for $w=1$, i.e.\ for $Q$. Similar unconditional multivariate results are
given in~\cite{mol:sch11}.

\begin{proposition}
\label{symmetry under H-gen}Let $S=\exp\left(  X\right)  $ be a martingale.
Then $S$ is self-dual if and only if for any stopping time $\tau\in[0,T]$ and
any non-negative Borel function $f$
\begin{equation}
E_{P^{w}}\left[  \left.  f\left(  X_{T}-X_{\tau}\right)  \right\vert
\mathcal{F}_{\tau}\right]  =E_{P^{1-w}}\left[  \left.  f\left(  X_{\tau}%
-X_{T}\right)  \right\vert \mathcal{F}_{\tau}\right]
\label{eq:symmetry-under-H-gen}%
\end{equation}
holds for at least one (and then necessarily for all) $w\in[0,1]$.
\end{proposition}

For the half measure we immediately obtain the following special case.

\begin{corollary}
\label{symmetry under H}Let $S=\exp\left(  X\right)  $ be a martingale. Then
$S$ is self-dual if and only if $X$ is conditionally symmetric with respect to
$H$.
\end{corollary}

\begin{proof}
[Proof of Proposition~\ref{symmetry under H-gen}]As a consequence of the
martingale property of $S$ and H\"{o}lder's inequality we have that both
$E_{P}\left[  \left.  e^{w(X_{T}-X_{\tau})}\right\vert \mathcal{F}_{\tau
}\right]  $, for all $w\in\lbrack0,1]$, as well as $\left.  |E_{P}\left[
e^{p(X_{T}-X_{\tau})}|\mathcal{F}_{\tau}\right]  \right\vert $, for all
complex $p=a+ib$ with $a\in\left[  0,1\right]  $, $\tau\in\lbrack0,T]$, are finite~a.s.

Let $S$ be self-dual, $w\in\lbrack0,1]$, so that~(\ref{self-duality lemma})
implies the following two equalities:
\begin{align}
E_{P}\left[  \left.  e^{w(X_{T}-X_{\tau})}\right\vert \mathcal{F}_{\tau
}\right]   &  =E_{P}\left[  \left.  e^{(1-w)(X_{T}-X_{\tau})}\right\vert
\mathcal{F}_{\tau}\right]  ,\label{eq:gen-h-im1}\\
E_{P}\left[  \left.  e^{(w+i\theta)(X_{T}-X_{\tau})}\right\vert \mathcal{F}%
_{\tau}\right]   &  =E_{P}\left[  \left.  e^{(1-w-i\theta)(X_{T}-X_{\tau}%
)}\right\vert \mathcal{F}_{\tau}\right]  , \label{eq:gen-h-im2}%
\end{align}
for $\theta\in\mathbb{R}$. By applying Bayes' formula we obtain
\begin{align*}
E_{P^{w}}\left[  \left.  e^{i\theta\left(  X_{T}-X_{\tau}\right)  }\right\vert
\mathcal{F}_{\tau}\right]   &  =\frac{E_{P}\left[  \left.  e^{\left(
w+i\theta\right)  \left(  X_{T}-X_{\tau}\right)  }\right\vert \mathcal{F}%
_{\tau}\right]  }{E_{P}\left[  \left.  e^{w\left(  X_{T}-X_{\tau}\right)
}\right\vert \mathcal{F}_{\tau}\right]  },\\
E_{P^{1-w}}\left[  \left.  e^{i\theta\left(  X_{\tau}-X_{T}\right)
}\right\vert \mathcal{F}_{\tau}\right]   &  =\frac{E_{P}\left[  \left.
e^{\left(  1-w-i\theta\right)  \left(  X_{T}-X_{\tau}\right)  }\right\vert
\mathcal{F}_{\tau}\right]  }{E_{P}\left[  \left.  e^{(1-w)\left(
X_{T}-X_{\tau}\right)  }\right\vert \mathcal{F}_{\tau}\right]  }\,,
\end{align*}
so that in view of~(\ref{eq:gen-h-im1},\ \ref{eq:gen-h-im2}) the
r.h.s.\ coincide and so do the l.h.s.\ Since the conditional characteristic
functions $\left(  X_{T}-X_{\tau}\right)  $ under $P^{w}$ coincide with the
ones of $\left(  X_{\tau}-X_{T}\right)  $ under $P^{1-w}$, we end up
with~(\ref{eq:symmetry-under-H-gen}) for the claimed cases.

On the other hand, for an arbitrary $w\in\lbrack0,1]$, the $P$-martingale
property of $S$, and by Bayes' formula we see that the l.h.s.\ (and hence the
r.h.s.) of the following equations coincide:
\begin{align}
E_{P^{w}}\left[  \left.  e^{-w(X_{T}-X_{\tau})}\right\vert \mathcal{F}_{\tau
}\right]   &  =E_{P}\left[  \left.  e^{w(X_{T}-X_{\tau})}\right\vert
\mathcal{F}_{\tau}\right]  ^{-1},\label{eq:gen-h-im3}\\
E_{P^{1-w}}\left[  \left.  e^{-w(X_{\tau}-X_{T})}\right\vert \mathcal{F}%
_{\tau}\right]   &  =\frac{E_{P}\left[  \left.  e^{X_{T}-X_{\tau}}\right\vert
\mathcal{F}_{\tau}\right]  }{E_{P}\left[  \left.  e^{(1-w)(X_{T}-X_{\tau}%
)}\right\vert \mathcal{F}_{\tau}\right]  }=E_{P}\left[  \left.  e^{(1-w)(X_{T}%
-X_{\tau})}\right\vert \mathcal{F}_{\tau}\right]  ^{-1}\,.
\label{eq:gen-h-im4}%
\end{align}
Furthermore, we have for all complex $p=a+ib$ with $a\in\lbrack0,1]$ that
\[
E_{P^{w}}\left[  \left.  e^{(p-w)(X_{T}-X_{\tau})}\right\vert \mathcal{F}%
_{\tau}\right]  =E_{P^{1-w}}\left[  \left.  e^{(w-p)(X_{T}-X_{\tau}%
)}\right\vert \mathcal{F}_{\tau}\right]  \,.
\]
Combining this equality with the fact that the r.h.s.\ of~(\ref{eq:gen-h-im3})
and~(\ref{eq:gen-h-im4}) coincide we obtain the equality of the l.h.s.\ of the
following two equations
\begin{align*}
E_{P^{w}}\left[  \left.  \frac{e^{p\left(  X_{T}-X_{\tau}\right)  }%
}{e^{w\left(  X_{T}-X_{\tau}\right)  }}\right\vert \mathcal{F}_{\tau}\right]
E_{P}\left[  \left.  e^{w\left(  X_{T}-X_{\tau}\right)  }\right\vert
\mathcal{F}_{\tau}\right]   &  =E_{P}\left[  \left.  e^{p(X_{T}-X_{\tau}%
)}\right\vert \mathcal{F}_{\tau}\right] \\
E_{P^{1-w}}\left[  \left.  \frac{e^{(1-p)\left(  X_{T}-X_{\tau}\right)  }%
}{e^{(1-w)\left(  X_{T}-X_{\tau}\right)  }}\right\vert \mathcal{F}_{\tau
}\right]  E_{P}\left[  \left.  e^{(1-w)\left(  X_{T}-X_{\tau}\right)
}\right\vert \mathcal{F}_{\tau}\right]   &  =E_{P}\left[  \left.
e^{(1-p)(X_{T}-X_{\tau})}\right\vert \mathcal{F}_{\tau}\right]  \,.
\end{align*}
The self-duality property then follows by using the equality of the r.h.s.\ of
the above equations and Lemma~\ref{Tehranchi lemma}.
\end{proof}

\bigskip

The following definition and proposition follow the unconditional versions
stated in~\cite{MoS}, see also~\cite{CL}.

\begin{definition}
An adapted positive process $S$ is \textbf{quasi self-dual }of order
$\alpha\in%
\mathbb{R}
$ if for any stopping time $\tau\leq T$ and any non-negative Borel function
$f$ it holds that%
\begin{equation}
\label{eq:def-quasi-equation}E_{P}\left[  \left.  f\left(  \frac{S_{T}%
}{S_{\tau}}\right)  \right\vert \mathcal{F}_{\tau}\right]  =E_{P}\left[
\left.  \left(  \frac{S_{T}}{S_{\tau}}\right)  ^{\alpha}f\left(  \frac
{S_{\tau}}{S_{T}}\right)  \right\vert \mathcal{F}_{\tau}\right]  .
\end{equation}
In particular, for all $\tau\leq T$
\[
E_{P}\left[  \left.  \left(  \frac{S_{T}}{S_{\tau}}\right)  ^{\alpha
}\right\vert \mathcal{F}_{\tau}\right]  =1.
\]

\end{definition}

\begin{proposition}
[Characterization of quasi self-duality]%
\label{characterization of quasi self-duality}$S$ is quasi self-dual of order
$\alpha\neq0$\textbf{\ }if and only if $S^{\alpha}$ is self-dual.
\end{proposition}

\begin{proof}
This follows by considering for each $f$ the functions $g$ defined by
$g\left(  x\right)  =f\left(  x^{\alpha}\right)  $, respectively $h$ given by
$h\left(  x\right)  =f\left(  x^{1/\alpha}\right)  $, $x>0$.
\end{proof}

\section{Quasi self-dual continuous martingales}

The goal of this section is to clarify the structure of quasi self-dual
processes in a continuous martingale setting which comprises some Brownian
motion-driven stochastic volatility models in financial applications.
Following~\cite{T} we assume throughout this section that every $(\mathcal{F}%
_{t})$-martingale is continuous. We refer to~\cite{RY} for all unexplained terminology.

For every continuous conditionally symmetric martingale $Y$, $Y_{0}=0$, such
that its stochastic exponential $\mathcal{E(}Y\mathcal{)}$ is a martingale,
one can define the change of measure%
\[
\frac{dQ}{dP}=\mathcal{E}\left(  Y\right)  _{T}=\exp\left(  Y_{T}-\frac{1}%
{2}\left[  Y\right]  _{T}\right)  .
\]

In the sequel, we assume that $Y$ is a continuous martingale with $Y_{0}=0$.
Let $X=Y-\frac{1}{2}\left[  Y\right]  $ and observe that $\left[  X\right]
=\left[  Y\right]  $, hence $Y=X+\frac{1}{2}\left[  X\right]  $. We assume
w.l.o.g. that $S_{0}=1$ and set $S=\exp(X)=\mathcal{E}\left(  Y\right)  $. By
Proposition~\ref{symmetry under H}, the self-duality of a martingale $S$ is
equivalent to the conditional symmetry of $X$ under the measure $H$. The next
result is significantly more difficult to prove.

\begin{theorem}
[Tehranchi \cite{T}, Theorem~3.1.]\label{symmetry and self-duality} The
continuous martingale $S$ is self-dual if and only if $S$ is of the form
$S=\mathcal{E}\left(  Y\right)  $ for a conditionally symmetric continuous
local martingale $Y$.
\end{theorem}

One particular problem in this context is that stochastic exponentials can be
strict local martingales in which case it would not be possible to use them as
density processes for the measure transform leading to the dual market in
financial interpretations. An example class of positive self-dual continuous
martingales is provided by stochastic exponentials of conditionally symmetric
$BMO$-martingales, see \cite{kaz} for a detailed exposition of $BMO$-theory.

\begin{proposition}
\label{symmetry and BMO}Let $Y$ be a continuous conditionally symmetric
martingale such that there exists a constant $C$ with%
\begin{equation}
\sup_{0\leq\tau\leq T}E\left[  |Y_{T}-Y_{\tau}| \vert\mathcal{F}_{\tau
}\right]  \leq C. \label{BMO condition}%
\end{equation}
Then $Y$ is a $BMO$-martingale and its stochastic Dol\'{e}ans-exponential
$\mathcal{E}\left(  \alpha Y\right)  $ is a martingale for each $\alpha
\in\mathbb{R}$. Moreover, the following two assertions are equivalent:
\end{proposition}

\begin{description}
\item[(i)] $S=\mathcal{E}\left(  Y\right)  $\textit{ is a positive self-dual
martingale which satisfies for some }$p>1$\textit{ the reverse H\"{o}lder
inequality }$R_{p}(P)$\textit{, or, equivalently the Muckenhoupt inequality
}$A_{q}(Q)$\textit{ for }$q=\left(  p+1\right)  /p$\textit{, both with the
same constant.}

\item[(ii)] $Y$\textit{ is a conditionally symmetric }$BMO$%
\textit{-martingale.}
\end{description}

\begin{proof}
Condition~(\ref{BMO condition}) implies that $Y\in BMO$. Consequently, by
Theorem 2.3 in~\cite{kaz}, $\mathcal{E}\left(  \alpha Y\right)  $ is a
martingale (and not a strict local martingale).

The uniform boundedness for all stopping times of the l.h.s.
of~(\ref{self-duality numeraire}) for $f(x)=\left\vert x\right\vert ^{p}$
corresponds to $R_{p}(P)$, and of the r.h.s. of~(\ref{self-duality numeraire})
for $f(x)=\left\vert x\right\vert ^{q}$ to $A_{q}(Q)$. The equivalence of (i)
and (ii) then follows from Theorems~2.3,~2.4 and 3.4 in~\cite{kaz}, together
with Theorem~\ref{symmetry and self-duality}.
\end{proof}

\medskip

The process $X=\log(S)$ is, in contrast to $Y$, typically not a martingale. As
$X=Y-\frac{1}{2}\left[  Y\right]  $, the \emph{minimal martingale measure
}$\widehat{P}$ (see~\cite{Schw}) for $X$ is well-defined if $\mathcal{E}%
\left(  \frac{1}{2}Y\right)  $ is a martingale, and has then the density
\[
\frac{d\widehat{P}}{dP}=\exp(\frac{1}{2}Y_{T}-\frac{1}{8}\left[  Y\right]
_{T})=\exp\left(  \frac{1}{2}X_{T}+\frac{1}{8}\left[  X\right]  _{T}\right)
.
\]

The \emph{minimal entropy martingale measure }$Q^{E}$ for $X$ is a martingale
measure which minimizes the relative entropy with respect to $P$ over all
martingale measures for $X$. It can be characterised as the martingale measure
for $X$ with finite relative entropy such that
\[
\frac{dQ^{E}}{dP}=\exp\left(  c+\int_{0}^{T}\eta_{t}\,dX_{t}\right)  ,
\]
where $\eta$ is a predictable process with the property that $\int\eta\,dX$ is
a $Q$-martingale for all martingale measures $Q$ with finite relative entropy,
see~\cite{GrR}. It follows from Corollary~\ref{symmetry under H} that under
mild conditions the measure $H$ with density%
\[
\frac{dH}{dP}=\exp\left(  c+\frac{1}{2}X_{T}\right)
\]

is a martingale measure for $X$. The preceding discussion shows that
typically, $H$ is the minimal entropy martingale measure with $\eta=1/2$. This
is in general different from the minimal martingale measure, see
\cite{GrR},\textit{ }p.~1036. Moreover, it is remarkable that $Q^{E}=H$ has
such a simple form, which has consequences for the structure of conditionally
symmetric martingales. In fact, for all $t\in\left[  0,T\right]  $ the measure
$H^{t}$ with density%
\begin{equation}
\frac{dH^{t}}{dP}=\exp\left(  c_{t}+\frac{1}{2}X_{t}\right)
\label{measures H^t}%
\end{equation}

is a martingale measure for $X$ on $\left[  0,t\right]  $ but the normalizing
constant $c_{t}$ depends of course on $t$.

\begin{definition}
Let $M$ be a continuous local martingale, and denote the right-continuous and
complete filtration generated by $M$ as $\mathbb{F}^{M}$. $M$ is said to have
the \textbf{PRP} (predictable representation property), if every
$\mathbb{F}^{M}$-adapted local martingale $N$ can be written as $N=N_{0}%
+\int\vartheta\,dM$ for some predictable, $M$-integrable process $\vartheta$.
\end{definition}

\begin{proposition}
Let $Y$ be a continuous $P$-martingale which is conditionally symmetric up to
$T$ and which has the PRP. Assume that $S=\mathcal{E}\left(  Y\right)  $ is a
martingale, and that the minimal martingale measure $\widehat{P}$ exists for
$X=Y-\frac{1}{2}\left[  Y\right]  $. Then $Y$ is a Gaussian martingale.
\end{proposition}

\begin{proof}
By~\cite{RY}, Exercise VIII 1.27., the fact that $Y$ has the PRP under $P$
implies that $X$ has the PRP under $\widehat{P}$. Moreover, existence of
$\widehat{P}$ implies existence of the probability measures $H^{t}$ as defined
in~(\ref{measures H^t}) because%
\[
E\left[  \exp\left(  \frac{1}{2}X_{t}\right)  \right]  \leq E\left[
\exp\left(  \frac{1}{2}X_{t}+\frac{1}{8}\left[  X\right]  _{t}\right)
\right]  =E\left[  E\left[  \left.  \frac{d\widehat{P}}{dP}\right\vert
\mathcal{F}_{t}\right]  \right]  =1.
\]
Since $Y$ is conditionally symmetric up to $T$, it follows from
Theorem~\ref{symmetry and self-duality} that $S$ is self-dual, and hence, by
Proposition~\ref{symmetry under H}, $X$ is a conditionally symmetric
martingale under each $H^{t}$. In particular, $H^{t}$ is a martingale measure
for $X$ on $[0,t]$.
The PRP implies by the second fundamental theorem of asset pricing, see
Theorem~1.17 of~\cite{CS}, that $\widehat{P}=H^{t}$ on $\mathcal{F}_{t}$ which
yields%
\[
\exp\left(  \frac{1}{2}Y_{t}-\frac{1}{8}\left[  Y\right]  _{t}\right)
=\exp\left(  \frac{1}{2}X_{t}+\frac{1}{8}\left[  X\right]  _{t}\right)
=\exp\left(  c_{t}+\frac{1}{2}X_{t}\right)  ,
\]
for all $t\leq T$. It follows that $\left[  Y\right]  =\left[  X\right]  $
must be deterministic, and therefore $Y$ is a Gaussian martingale.
\end{proof}

\medskip

The next result completely characterises quasi self-dual continuous
semimartingales in terms of conditional symmetry.

\begin{theorem}
\label{quasi self-duality continuous martingales}A continuous positive
semimartingale $S$ is quasi self-dual of non-vanishing order $\alpha
=1-2\kappa$, if and only if $S^{\alpha}$ is a martingale and $S=e^{\kappa
\left[  M\right]  }\mathcal{E}\left(  M\right)  $ for a continuous
conditionally symmetric local martingale $M$. For $\alpha=0$ we assume in
addition that $S=\exp(M)$ for an integrable process $M$. In that case, $S$ is
quasi self-dual of order zero if and only if $M$ is a continuous conditionally
symmetric martingale.
\end{theorem}

\begin{proof}
For $\alpha\neq0$, $S$ is quasi self-dual if and only if $S^{\alpha}$ is
self-dual for some $\alpha$, hence in particular a positive continuous
martingale. We can then write by Theorem~\ref{symmetry and self-duality}%
\[
S^{\alpha}=\mathcal{E}\left(  \alpha M\right)  =\exp\left(  \alpha M-\frac
{1}{2}\alpha^{2}\left[  M\right]  \right)  =e^{\alpha\kappa\left[  M\right]
}\mathcal{E}\left(  M\right)  ^{\alpha}%
\]
for some conditionally symmetric local martingale $M$. On the other hand, if
$S=e^{\kappa\left[  M\right]  }\mathcal{E}\left(  M\right)  $, we have%
\[
S^{\alpha}=e^{\alpha\kappa\left[  M\right]  }\mathcal{E}\left(  M\right)
^{\alpha}=\exp\left(  \alpha M+\left(  \kappa-\frac{1}{2}\right)
\alpha\left[  M\right]  \right)  .
\]
Since $S^{\alpha}$ is a positive martingale it follows that $S^{\alpha
}=\mathcal{E}\left(  N\right)  $ for a continuous local martingale $N$. The
uniqueness of the canonical semimartingale decomposition implies $N=\alpha M$
which implies that
\[
\left(  \kappa-\frac{1}{2}\right)  \alpha=-\frac{1}{2}\alpha^{2}.
\]
\noindent In the case when $\alpha\neq0$, dividing by $\alpha$ yields the
result. If $\alpha=0$, we start by assuming that $S$ is quasi self-dual of
order $\alpha=0$. For an arbitrary non-negative Borel function $f$, define
$g=f\circ\log$. By assumption we have for any stopping time $\tau\in\left[
0,T\right]  $
\begin{align*}
E\left[  \left.  f\left(  M_{T}-M_{\tau}\right)  \right\vert \mathcal{F}%
_{\tau}\right]   &  =E\left[  \left.  g\left(  \exp(M_{T}-M_{\tau})\right)
\right\vert \mathcal{F}_{\tau}\right] \\
&  =E\left[  \left.  g\left(  \exp(M_{\tau}-M_{T})\right)  \right\vert
\mathcal{F}_{\tau}\right]  =E\left[  \left.  f\left(  M_{\tau}-M_{T}\right)
\right\vert \mathcal{F}_{\tau}\right]  ,
\end{align*}
for an arbitrary non-negative Borel function $f$, i.e.\ $M$ is conditionally
symmetric combined with the integrability assumption, also a martingale, and
it is clearly continuous. Furthermore, $S=\exp(M)=e^{\frac{1}{2}\left[
M\right]  }\mathcal{E}\left(  M\right)  $ holds.

Conversely, if $M=\log(S)$ is a continuous conditionally symmetric martingale,
then for all non-negative Borel functions $f$ define $g=f\circ\exp$ so that
\begin{align*}
E\left(  \left.  f(\exp\left(  M_{T}-M_{\tau}\right)  )\right\vert
\mathcal{F}_{\tau}\right)   &  =E\left(  \left.  g\left(  M_{T}-M_{\tau
}\right)  \right\vert \mathcal{F}_{\tau}\right) \\
&  =E\left(  \left.  g\left(  M_{\tau}-M_{T}\right)  \right\vert
\mathcal{F}_{\tau}\right)  = E\left(  \left.  f\left(  \exp(M_{\tau}%
-M_{T})\right)  \right\vert \mathcal{F}_{\tau}\right)  ,
\end{align*}
which implies the quasi self-duality of order zero.
\end{proof}

\begin{example}
[Geometric Brownian motion]The results presented lead to the following view of
the symmetries of geometric Brownian motion. Let $Y=\sigma W$ for a standard
Brownian motion $W$ and $\sigma>0$. Since $Y$ is a continuous conditionally
symmetric martingale, Theorem~\ref{symmetry and self-duality} yields that
\[
\mathcal{E}\left(  Y\right)  =\exp\left(  Y-\frac{1}{2}\left[  Y\right]
\right)  =\exp\left(  \sigma W-\frac{1}{2}\sigma^{2}t\right)
\]
is a self-dual process. Denoting by $\lambda$ a shift parameter, we consider
the process
\[
S=\exp\left(  \sigma W-\frac{1}{2}\sigma^{2}t+\lambda t\right)  =\exp\left(
\kappa\lbrack Y]\right)  \mathcal{E}\left(  Y\right)  \,,\quad\text{where
}\quad\kappa=\frac{\lambda}{\sigma^{2}}\,.
\]
By Theorem~\ref{quasi self-duality continuous martingales} $S$ is quasi
self-dual of order zero if and only if $\lambda=\frac{1}{2}\sigma^{2}$ (since
we have the ordinary exponent of a continuous conditionally symmetric
martingale) and, in view of the fact that $S^{\alpha}$, $\alpha=1-2\kappa
=1-\frac{2\lambda}{\sigma^{2}}$, is a martingale, it is quasi self-dual of
order $\alpha$, cf.\ e.g.~\cite{car:cho97,CL}.
\end{example}

\section{Ocone martingales and strong self-duality}

\label{sec:Ocone}

\medskip In this section we discuss a connection between Ocone martingales and
a strong version of self-duality of their associated stochastic exponentials,
as motivated by the discussion in \cite{T}. However, the previously introduced
conditional notions of symmetry resp.\ self-duality are not quite fitting for
such a discussion as Ocone martingales in particular enjoy a stronger notion
of symmetry.

\begin{definition}
Let $M$ be a continuous $P$-martingale vanishing at zero and such that
$\left[  M\right]  _{\infty}=\infty$, and consider its Dambis-Dubins-Schwartz
(DDS) representation $M=B_{\left[  M\right]  }$. The process $M$ is called an
\textbf{Ocone martingale} if $B$ and $\left[  M\right]  $ are independent.
\end{definition}

It has been proved in~\cite{VY} that if a martingale is Ocone and has the PRP,
then it is Gaussian. A more interesting example of an Ocone martingale is
given by the solution of the stochastic differential equation%
\begin{align*}
dM_{t}  &  =V_{t}\,dB_{t},\\
dV_{t}  &  =-\mu V_{t}\,dt+\sqrt{V_{t}}\,dW_{t},
\end{align*}

where $\mu>0$ and $B$, $W$ are two independent Brownian motions. This follows
by~\cite{BNS}, Ch.~2,~Th.~2.6, since $\left[  M\right]  =\int V^{2}\,dt$ is
independent of $B$. Moreover, L\'{e}vy's stochastic area process is also an
Ocone martingale, see \cite{VY}.

\begin{definition}
An adapted process $X$ is process symmetric if $X\sim-X$ (i.e.\ the finite
dimensional distributions of $X$ and $-X$ are the same). In particular, for
semimartingales $X$ with $X_{0}=0$ this is equivalent to%
\[
E\left[  \exp\left(  i\int_{0}^{T}\theta_{t}\,dX_{t}\right)  \right]
=E\left[  \exp\left(  -i\int_{0}^{T}\theta_{t}\,dX_{t}\right)  \right]
\quad\forall\theta\in\mathcal{S},
\]

where $\mathcal{S}$ denotes the space of deterministic and bounded Borel
functions on $\left[  0,T\right]  $.
\end{definition}

\begin{remark}
Ocone martingales are always process symmetric, see Tehranchi \cite{T}.
\end{remark}

It is important to stress that different symmetry concepts are not equivalent.
Since for example conditional symmetry implies the martingale property for
integrable processes, we have that an integrable process symmetric $X$ which
is not a martingale cannot be conditionally symmetric. For example, if $Z$ is
a symmetric integrable random variable, then the process $\left(  Zt\right)
_{t\in\left[  0,T\right]  }$ is still process symmetric but not a martingale.
Less obvious is that there are also process symmetric martingales which are
not conditionally symmetric.

\begin{example}
\label{ex:counter}$\left(  i\right)  $ The martingale $M=\int B^{2}\,dB$ is
process symmetric since%
\[
-\int B^{2}\,dB=\int\left(  -B\right)  ^{2}\,d\left(  -B\right)  .
\]

Since
\[
B= \int\left(  \frac{d[M]}{dt}\right)  ^{-\frac{1}{2}}dM\,,
\]
we have that the Brownian filtration $\mathbb{F}=\left(  \mathcal{F}%
_{t}\right)  $ equals the filtration $\mathbb{F}^{M}$ generated by $M$.
Moreover, $M$ has the PRP, but is non-Gaussian, and hence not Ocone.

\medskip

$\left(  ii\right)  $ It is worth noting, in light of Theorem
\ref{Self-duality and Ocone}, that the stochastic exponential $\mathcal{E}%
\left(  M\right)  $ is a strict local martingale. This follows e.g.\ by
Corollary 2.2 of \cite{MU} since the with $M$ associated auxiliary diffusion%
\[
d\widetilde{Y}_{t}=\widetilde{Y}_{t}^{2}\,dt+dB_{t}%
\]
does explode.

\medskip

$\left(  iii\right)  $ However, $M$ is not conditionally symmetric. Choose
$0<t<T$ and assume by means of contradiction that $M$ is conditionally
symmetric. In particular,
\[
E[(M_{T}-M_{t})^{3}|\mathcal{F}_{t}]=E[(M_{t}-M_{T})^{3}|\mathcal{F}_{t}]\,,
\]
since the symmetry is satisfied by the positive and the negative part of the
third conditional moment, so that
\[
E[(M_{T}-M_{t})^{3}|\mathcal{F}_{t}]=0
\]
holds a.s., where we have used that $E\left[  |M_{T}-M_{t}|^{3}\right]
<\infty$. However, noting that $W_{s}=B_{t+s}-B_{t}$ defines a Brownian motion
independent of $\mathcal{F}_{t}$ and that $B_{t}$ is $\mathcal{F}_{t}%
$-measurable, we can write
\[
M_{T}-M_{t}=\int_{0}^{T-t}W_{s}dW_{s}+B_{t}^{2}W_{T-t}+B_{t}(W_{T-t}%
^{2}-(T-t))\,.
\]
By a straightforward but lengthy calculation, the third conditional moment of
this martingale increment can, for $t<T$, be written as a nontrivial
polynomial in $B_{t}$, which, for $t>0$, will not take a.s.\ only values in
the roots of the polynomial, i.e.\ we end up with a contradiction.
\end{example}

\medskip

It is observed by Tehranchi~\cite{T} that continuous Ocone martingales are
conditionally symmetric with respect to deterministic times. The next result
shows that this is still true for bounded stopping times.

\begin{lemma}
\label{le:ocone-io} A continuous Ocone martingale $M$ with natural filtration
$\mathbb{F=}(\mathcal{F}_{t})$ is conditionally symmetric.
\end{lemma}

\begin{proof}
Since $M$ is Ocone we have $M=\beta_{\left[  M\right]  }$ for a Brownian
motion $\beta$ (with natural filtration $\mathbb{B=}\left(  \mathcal{B}%
_{t}\right)  $) being independent of $[M]$. Denote by $\mathbb{G=}\left(
\mathcal{G}_{t}\right)  $ the right-continuous enlargement of the filtration
$\mathcal{B}_{t}\vee\sigma([M])$. Following~\cite[p.~129]{dub:em:yor93} we
have that for $\tau$ being an $\mathbb{F}$-stopping time it follows that
$[M]_{\tau}$ is a $\mathbb{G}$-stopping time. Indeed, by denoting $A_{t}%
=\inf\{s:[M]_{s}>t\}$ we have
\[
\{[M]_{\tau}\leq t\}=\{\tau\leq A_{t}\}\in\mathcal{F}_{A_{t}}\subset
\bigcap_{\varepsilon>0}\sigma(M^{A_{t+\varepsilon}})\subset\bigcap
_{\varepsilon>0}[\mathcal{B}_{t+\varepsilon}\vee\sigma([M])]=\mathcal{G}%
_{t}\,,
\]
where in our case $\tau\leq T<\infty$ and $[M]$ is continuous, so that we end
up with a finite stopping time. As in the proof of Lemma~2 on p.~129
in~\cite{dub:em:yor93} the Ocone property and Lemma~1 on p.~129
in~\cite{dub:em:yor93} imply that $\beta$ is a $\mathbb{G}$-Brownian motion.
By~\cite[Th.~13.11]{kalle} $\beta_{u}^{\prime}=\beta_{u+[M]_{\tau}}%
-\beta_{\lbrack M]_{\tau}}$ is a Brownian motion independent of $\mathcal{G}%
_{[M]_{\tau}}$. Hence, for $\tau\leq T<\infty$ we have for any non-negative
Borel function $f$
\begin{align*}
E[f(M_{T}-M_{\tau})|\mathcal{F}_{\tau}]  &  =E[E[f(\beta_{\lbrack M]_{T}%
}-\beta_{\lbrack M]_{\tau}})|\mathcal{G}_{[M]_{\tau}}]|\mathcal{F}_{\tau}]\\
&  =E[E[f(\beta_{\lbrack M]_{\tau}}-\beta_{\lbrack M]_{T}})|\mathcal{G}%
_{[M]_{\tau}}]|\mathcal{F}_{\tau}]=E[f(M_{\tau}-M_{T})|\mathcal{F}_{\tau}]\,,
\end{align*}
so that $M$ is conditionally symmetric.
\end{proof}

\medskip

We assume w.l.o.g. that $S_{0}=1$ and set $S=\exp(X)=\mathcal{E}\left(
Y\right)  $ where $X$ (and then $Y$) is a continuous semimartingale.

Recall that $\mathcal{S}$ is the space of deterministic and bounded Borel
functions on $\left[  0,T\right]  $. For $\phi\in\mathcal{S}$, we set
$X^{\phi}=\int\phi\,dY-\frac{1}{2}\int\phi^{2}\,d\left[  Y\right]  $, and
$S^{\phi}:=\exp\left(  X^{\phi}\right)  =\mathcal{E}\left(  \int
\phi\,dY\right)  $. In the case that all $S^{\phi}$ are martingales, we define
dual probability measures $Q^{\phi}$ via%
\[
\frac{dQ^{\phi}}{dP}=S_{T}^{\phi}.
\]

\begin{definition}
We say that $S=\mathcal{E}(Y)$, for a continuous martingale $Y$, is strongly
self-dual if for all $\phi\in\mathcal{S}$, the $S^{\phi}$ are martingales and
that with equality in distribution as a process living on $\left[  0,T\right]
$,
\[
\left\{  S^{\phi}\,,P\right\}  \sim\left\{  \frac{1}{S^{\phi}}\,,Q^{\phi
}\right\}  \,.
\]

\end{definition}

In the case when $S$ is strongly self-dual then, by choosing $\phi=1$, it also
implies an unconditional form of the self-duality property known as put-call
symmetry, which is the most frequently used definition in the previous
literature.

\begin{lemma}
\label{iff cond for g.s.d.}$S$ is strongly self-dual if and only if for all
$\phi\in\mathcal{S}$, the $S^{\phi}$ are martingales and
\begin{equation}
\left\{  X^{\phi}\,,P\right\}  \sim\left\{  -X^{\phi}\,, Q^{\phi}\right\}  \,.
\label{symmetry X phi}%
\end{equation}

\end{lemma}


\begin{proof}
Note that the martingale assumptions are the same. Assume that $S$ is strongly
self-dual. For an arbitrary non-negative functional $F$ define $G$ via
\[
G:=F\circ\log\,.
\]
Then
\begin{align*}
E_{P}\left[  F(X_{t}^{\phi}\,,0\leq t\leq T)\right]   &  =E_{P}\left[
G(S_{t}^{\phi}\,,0\leq t\leq T)\right]  =E_{Q^{\phi}}\left[  G\left(  \frac
{1}{S_{t}^{\phi}}\,,0\leq t\leq T\right)  \right] \\
&  =E_{Q^{\phi}}\left[  F(-X_{t}^{\phi}\,,0\leq t\leq T)\right]  \,,
\end{align*}
so that~(\ref{symmetry X phi}) follows. The converse direction follows similarly.
\end{proof}

\medskip

For the record, we state the next proposition which is analogous to
Corollary~\ref{symmetry under H}. If for $\phi\in\mathcal{S}$, $\exp\left(
X^{\phi}\right)  $ is a $P$-martingale, then%
\[
c_{\phi}=E\left[  \exp\left(  \frac{1}{2}X_{T}^{\phi}\right)  \right]
<\infty,
\]

so we can define probability measures $H^{\phi}$ (analogous to the half
measure $H$) via%
\[
\frac{dH^{\phi}}{dP}=c_{\phi}^{-1}\exp\left(  \frac{1}{2}X_{T}^{\phi}\right)
.
\]

\begin{proposition}
$S$ is strongly self-dual if and only if for all $\phi\in\mathcal{S}$ we have
that the $S^{\phi}$ are martingales and
%
the $X^{\phi}$ are process symmetric under the measures $H^{\phi}$.
\end{proposition}

\begin{proof}
The strong self-duality implies the martingale assumptions and with
Lemma~\ref{iff cond for g.s.d.} it implies for arbitrary $\lambda$, $\phi
\in\mathcal{S}$ that
\[
c_{\phi}^{-1}E_{P}\left[  \exp\left(  \frac{1}{2}X_{T}+i\int_{0}^{T}%
\lambda_{t}\,dX_{t}^{\phi}\right)  \right]  =c_{\phi}^{-1}E_{Q^{\phi}}\left[
\exp\left(  \frac{1}{2}(-X_{T})+i\int_{0}^{T}\lambda_{t}\,d(-X_{t}^{\phi
})\right)  \right]  \,,
\]
while from the definitions of $H^{\phi}$ and $Q^{\phi}$
\[
E_{H^{\phi}}\left[  \exp\left(  i\int_{0}^{T}\lambda_{t}\,dX_{t}^{\phi
}\right)  \right]  =c_{\phi}^{-1}E_{P}\left[  \exp\left(  \frac{1}{2}%
X_{T}+i\int_{0}^{T}\lambda_{t}\,dX_{t}^{\phi}\right)  \right]
\]
and
\begin{align*}
E_{H^{\phi}}\left[  \exp\left(  -i\int_{0}^{T}\lambda_{t}\,dX_{t}^{\phi
}\right)  \right]   &  =c_{\phi}^{-1}E_{P}\left[  \exp\left(  \frac{1}{2}%
X_{T}+i\int_{0}^{T}\lambda_{t}\,d(-X_{t}^{\phi})\right)  \right] \\
&  =c_{\phi}^{-1}E_{Q^{\phi}}\left[  \exp\left(  \frac{1}{2}(-X_{T})+i\int
_{0}^{T}\lambda_{t}\,d(-X_{t}^{\phi})\right)  \right]  \,.
\end{align*}
Since $\lambda$ and $\phi$ were arbitrarily chosen, we end up with the process
symmetries of the processes $X^{\phi}$ under $H^{\phi}$.

Conversely, the process symmetries imply for arbitrary $\lambda$ (and $\phi
$)$\in\mathcal{S}$ that
\[
c_{\phi}E_{H^{\phi}}\left[  \exp\left(  -\frac{1}{2}X_{T}+i\int_{0}^{T}%
\lambda_{t} \,dX_{t}^{\phi}\right)  \right]  =c_{\phi}E_{H^{\phi}}\left[
\exp\left(  -\frac{1}{2}(-X_{T})+i\int_{0}^{T}\lambda_{t} \,d(-X_{t}^{\phi
})\right)  \right]  \,,
\]
while again by the definitions of $H^{\phi}$ and $Q^{\phi}$
\[
E_{P}\left[  \exp\left(  i\int_{0}^{T}\lambda_{t} \,dX_{t}^{\phi}\right)
\right]  =c_{\phi}E_{H^{\phi}}\left[  \exp\left(  -\frac{1}{2}X_{T}+i\int
_{0}^{T}\lambda_{t} \,dX_{t}^{\phi}\right)  \right]
\]
and
\begin{align*}
E_{Q^{\phi}}\left[  \exp\left(  -i\int_{0}^{T}\lambda_{t} \,dX_{t}^{\phi
}\right)  \right]   &  =E_{P}\left[  \exp\left(  X_{T}+i\int_{0}^{T}%
\lambda_{t} \,d(-X_{t}^{\phi})\right)  \right] \\
&  =c_{\phi}E_{H^{\phi}}\left[  \exp\left(  -\frac{1}{2}(-X_{T})+i\int_{0}%
^{T}\lambda_{t} \,d(-X_{t}^{\phi})\right)  \right]  \,.
\end{align*}
Hence, in view of the imposed martingale assumptions, the strong self-duality
now follows by Lemma~\ref{iff cond for g.s.d.}.
\end{proof}

\medskip

In the following result we show that the Ocone property translates one-to-one
via lifting by stochastic exponentiation into the strong self-duality property.

\begin{theorem}
\label{Self-duality and Ocone} A continuous martingale $Y$ is an Ocone
martingale if and only if $\mathcal{E}\left(  Y\right)  $ is strongly self-dual.
\end{theorem}

\begin{proof}
Let first $Y$ be a continuous Ocone martingale, and set for $\phi
\in\mathcal{S}$%
\[
\widetilde{Y}=Y-\int\phi\,d\left[  Y\right]  .
\]
Here we suppress the dependency of $\widetilde{Y}$ on $\phi$ for ease of notation.

By Theorem 1 and Comment 2 of Vostrikova \& Yor \cite{VY}, it holds that $Y$
is a continuous Ocone martingale if and only if

\medskip

$\left(  i\right)  $ For all $\phi\in\mathcal{S}$,
\begin{equation}
\left\{  Y\,,P\right\}  \sim\left\{  \widetilde{Y}\,,Q^{\phi}\right\}  \,,
\label{VY Ocone}%
\end{equation}

or equivalently,%
\begin{equation}
\left\{  \left[  Y\right]  ,P\right\}  \sim\left\{  \left[  Y\right]  \,,
Q^{\phi}\right\}  ; \label{VY Ocone 2}%
\end{equation}

note that $\left[  Y\right]  =\left[  \widetilde{Y}\right]  $.

\medskip

$\left(  ii\right)  $ $\mathcal{E}\left(  \int\phi\,dY\right)  $ is a
martingale for all $\phi\in\mathcal{S}$.

\medskip

We now show that $\left\{  X^{\phi}\,,P\right\}  \sim\left\{  -X^{\phi
}\,,Q^{\phi}\right\}  $ for all $\phi\in\mathcal{S}$, which implies (in view
of~$(ii)$) by Lemma \ref{iff cond for g.s.d.} strong self-duality of
$S=\mathcal{E}\left(  Y\right)  $. First note that $\tilde Y$ under $Q^{\phi}$
is also Ocone and thus, in particular process symmetric. Furthermore, we have
for all $\lambda\,,\psi\in\mathcal{S}$ by the aforementioned properties of
Ocone martingales, cf.\ also Lemma~2.5 of~\cite{O}, that
\begin{align*}
&  E\left[  \exp\left(  i\int_{0}^{T}\lambda_{t}\,dX_{t}^{\phi}\right)
\right]  =E\left[  \exp\left(  i\int_{0}^{T}\lambda_{t}\,d\left(  \int_{0}%
^{t}\phi_{s}\,dY_{s}-\frac{1}{2}\int_{0}^{t}\phi_{s}^{2}\,d\left[  Y\right]
_{s}\right)  \right)  \right] \\
&  =E_{Q^{\phi}}\left[  \exp\left(  i\int_{0}^{T}\lambda_{t}\,d\left(
\int_{0}^{t}\phi_{s}\,d\widetilde{Y}_{s}-\frac{1}{2}\int_{0}^{t}\phi_{s}%
^{2}\,d\left[  \widetilde{Y}\right]  _{s}\right)  \right)  \right] \\
&  = E_{Q^{\phi}}\left[  \exp\left(  i\int_{0}^{T}\lambda_{t}\,d\left(
\int_{0}^{t}\phi_{s}\,d\left(  -\widetilde{Y}_{s}\right)  -\frac{1}{2}\int
_{0}^{t}\phi_{s}^{2}\,d\left[  \widetilde{Y}\right]  _{s}\right)  \right)
\right] \\
&  =E_{Q^{\phi}}\left[  \exp\left(  i\int_{0}^{T}\lambda_{t}\,d\left(
-\left(  \int_{0}^{t}\phi_{s}\,d Y_{s} -\int_{0}^{t}\phi_{s}^{2}\,d\left[
Y\right]  _{s}+\frac{1}{2}\int_{0}^{t}\phi_{s}^{2}\,d\left[  Y\right]
_{s}\right)  \right)  \right)  \right] \\
&  =E_{Q^{\phi}}\left[  \exp\left(  i\int_{0}^{T}\lambda_{t}\,d\left(
-\left(  \int_{0}^{t}\phi_{s}\,dY_{s}-\frac{1}{2}\int_{0}^{t}\phi_{s}%
^{2}\,d\left[  Y\right]  _{s}\right)  \right)  \right)  \right] \\
&  =E_{Q^{\phi}}\left[  \exp\left(  -i\int_{0}^{T}\lambda_{t}\,dX_{t}^{\phi
}\right)  \right]
\end{align*}

which proves the claim and hence the first implication.

\medskip

\noindent As for the other direction, we will show property (\ref{VY Ocone 2})
for all $\phi\in\mathcal{S}$. Let $\mathcal{E}\left(  \int\phi\,dY\right)
=\exp(X^{\phi})$ where $X^{\phi}=\int\phi\,dY-\frac{1}{2}\int\phi
^{2}\,d\left[  Y\right]  $, and define probability measures $Q^{\phi}$ via
\[
dQ^{\phi}/dP=\mathcal{E}\left(  \int\phi\,dY\right)  _{T}=\exp\left(
X_{T}^{\phi}\right)  .
\]

We assume first that $\phi$ is bounded away from zero, and that $Y$ is a
square-integrable martingale, i.e. $E\left[  Y_{t}^{2}\right]  <\infty$ for
all $t\geq0$. Note that $\left[  X^{\phi}\right]  =\int\phi^{2}\,d\left[
Y\right]  $ and therefore $\left[  Y\right]  =\int\phi^{-2}\,d\left[  X^{\phi
}\right]  $. We have for every non-negative functional $F$, with $F\left(
U\right)  $ denoting $F\left(  U_{t};0\leq t\leq T\right)  $ for any
stochastic process $U$,
\[
E_{Q^{\phi}}\left[  F\left(  \left[  Y\right]  \right)  \right]  =E_{P}\left[
\exp\left(  X_{T}^{\phi}\right)  F\left(  \int\phi^{-2}\,d\left[  X^{\phi
}\right]  \right)  \right]  \newline=E_{Q^{\phi}}\left[  F\left(  \int
\phi^{-2}\,d\left[  X^{\phi}\right]  \right)  \right]  .
\]

On the other hand,%
\[
E_{P}\left[  F\left(  \left[  Y\right]  \right)  \right]  =E_{P}\left[
F\left(  \int\phi^{-2}\,d\left[  X^{\phi}\right]  \right)  \right]  .
\]
Hence the strong self-duality of $S$ implies by
Lemma~\ref{iff cond for g.s.d.} that $\left[  Y\right]  $ under $Q^{\phi}$ has
the same law as $\left[  Y\right]  $ under $P$, for all $\phi\in\mathcal{S}$
which are bounded away from zero. Equivalently, for all such $\phi$; $N\in%
\mathbb{N}
$ arbitrary; $0\leq t_{1}\leq...\leq t_{N}\leq T$; and arbitrary
$u=(u_{1},...,u_{N})\in%
\mathbb{R}
^{N}$ we set%
\[
\Psi(u)=\exp\left(  i\left(  u_{1}\left[  Y\right]  _{t_{1}}+...+u_{N}\left[
Y\right]  _{t_{N}}\right)  \right)  ,
\]
and have that\noindent%
\begin{align}
E\left[  \Psi(u)\exp\left(  \int_{0}^{T}\phi_{t}\,dY_{t}-\frac{1}{2}\int
_{0}^{T}\phi_{t}^{2}\,d\left[  Y\right]  _{t}\right)  \right]   &
=E_{Q^{\phi}}\left[  \Psi(u)\right] \nonumber\\
&  =E\left[  \Psi(u)\right]  . \label{Psi equivalenz}%
\end{align}

Let now $\phi\in\mathcal{S}$ be arbitrary, i.e., in particular $\phi$ may
vanish on a set $\Gamma\subset\left[  0,T\right]  $. We denote by
$\phi^{\left(  n\right)  }\in\mathcal{S}$ functions which coincide with $\phi$
as long as $\left\vert \phi_{t}\right\vert \geq1/n$, and which equal $1/n$ if
$\left\vert \phi_{t}\right\vert \leq1/n$, so that $\phi^{\left(  n\right)
}\rightarrow\phi$ pointwise. By dominated convergence for stochastic integrals
(see \cite{RY}, Theorem IV.2.12), it follows that in probability,%
\[
U_{n}:=\mathcal{E}\left(  \int\phi^{(n)}\,dY\right)  _{T}\rightarrow
\mathcal{E}\left(  \int\phi\,dY\right)  _{T}.
\]
We will now show that the family $\left\{  U_{n}\right\}  _{n\in%
\mathbb{N}
}$ is uniformly integrable. For this, it suffices to show that%
\[
\sup_{n}E\left[  U_{n}\log\left(  U_{n}\right)  \right]  <\infty.
\]
Indeed, using that so far $Y$ was assumed to be a square-integrable
martingale, and therefore $E\left[  \left[  Y\right]  _{t}\right]  <\infty$
for all $t\geq0$ (see \cite{P}, Corollary II.6.3),
\begin{align*}
E\left[  U_{n}\log\left(  U_{n}\right)  \right]   &  =E_{Q^{\phi^{\left(
n\right)  }}}\left[  \int_{0}^{T}\phi_{t}^{(n)}\,dY_{t}-\frac{1}{2}\int
_{0}^{T}\left(  \phi_{t}^{(n)}\right)  ^{2}\,d\left[  Y\right]  _{t}\right] \\
&  =E_{Q^{\phi^{\left(  n\right)  }}}\left[  \int_{0}^{T}\phi_{t}%
^{(n)}\,dY_{t}-\int_{0}^{T}\left(  \phi_{t}^{(n)}\right)  ^{2}\,d\left[
Y\right]  _{t}+\frac{1}{2}\int_{0}^{T}\left(  \phi_{t}^{(n)}\right)
^{2}\,d\left[  Y\right]  _{t}\right] \\
&  =E_{Q^{\phi^{\left(  n\right)  }}}\left[  \frac{1}{2}\int_{0}^{T}\left(
\phi_{t}^{(n)}\right)  ^{2}\,d\left[  Y\right]  _{t}\right] \\
&  =E\left[  \frac{1}{2}\int_{0}^{T}\left(  \phi_{t}^{(n)}\right)
^{2}\,d\left[  Y\right]  _{t}\right]  ,
\end{align*}
since by what has already been proved, $\left[  Y\right]  $ has the same
distribution under both $Q^{\phi^{(n)}}$ and $P$ since the $\phi^{(n)}$ are
bounded away from zero. We have, since $Y$ was assumed to be
square-integrable,%
\[
E\left[  \frac{1}{2}\int_{0}^{T}\left(  \phi_{t}^{(n)}\right)  ^{2}\,d\left[
Y\right]  _{t}\,\right]  \leq\mathrm{const.}E\left[  \left[  Y\,\right]
_{T}\,\right]  <\infty.
\]

Hence the $U_{n}$ are uniformly integrable, and we conclude by
(\ref{Psi equivalenz}) that%
\begin{align*}
&  E\left[  \Psi(u)\exp\left(  \int_{0}^{T}\phi_{t}\,dY_{t}-\frac{1}{2}%
\int_{0}^{T}\phi_{t}^{2}\,d\left[  Y\right]  _{t}\right)  \right] \\
&  =\lim_{n}E\left[  \Psi(u)\exp\left(  \int_{0}^{T}\phi_{t}^{\left(
n\right)  }\,dY_{t}-\frac{1}{2}\int_{0}^{T}\left(  \phi_{t}^{\left(  n\right)
}\right)  ^{2}\,d\left[  Y\right]  _{t}\right)  \right] \\
&  =E\left[  \Psi(u)\right]  .
\end{align*}

Therefore $\left[  Y\right]  $ has the same distribution under both $Q^{\phi}$
and $P$ for all $\phi\in\mathcal{S}$. It follows by the aforementioned result
of \cite{VY} that $Y$ is an Ocone martingale.

In the case that $Y$ is a continuous martingale, not necessarily
square-integrable, the result follows by localization: there is an increasing
sequence of stopping times $\left(  T_{n}\right)  $ such that $Y^{T_{n}}$ is
bounded, e.g. $T_{n}=\inf\left\{  t:\left\vert Y_{t}\right\vert =n\right\}  $.
Therefore the previous result applies on $\left[  0,T_{n}\right]  $, each $n$.
Letting $n$ tend to infinity yields then the general result.
\end{proof}

\bigskip

This result should be seen in the context of Tehranchi's \cite{T} Theorem as
given in Theorem \ref{symmetry and self-duality}. In both results, a certain
symmetry property (conditional respectively Ocone symmetry) translates into a
self-duality property (conditional respectively strong) of the associated
stochastic exponential. The structure of the proofs, however, is completely
different. While Tehranchi's proof rests on his characterization that for a
continuous local martingale $Y$ conditional symmetry is equivalent to the
property that $Y_{T}$ given $\mathcal{F}_{t}\vee\sigma\left(  \left[
Y\right]  _{T}\right)  $ is normally distributed with expectation $Y_{t}$ and
variance $\left[  Y\right]  _{T}-\left[  Y\right]  _{t}$ for all $0\leq t\leq
T$, we work with properties of Ocone martingales as developed in \cite{O} and
\cite{VY}. Again, it is still open whether there exists \emph{any }non-Ocone
conditionally symmetric martingale.

The notion of strong self-duality is justified by an economic interpretation,
namely that the distribution of the price process $S^{\phi}$ for arbitrary
parameter function $\phi\in\mathcal{S}$ remains invariant under the dual
market transformation. Furthermore, it is shown in Theorem
\ref{Self-duality and Ocone} that it is equivalent to the Ocone property of
its stochastic logarithm. The relationsship between conditional and strong
self-duality is beyond the scope of this paper.

\bigskip

\begin{acknowledgement}
The authors are grateful to Zhanyu Chen, Ilya Molchanov, Marcel Nutz, Kaspar
Stucki and Michael Tehranchi for helpful discussions and hints. Michael
Schmutz was supported by Swiss National Fund Project Nr. 200021-126503.
Furthermore, financial support by EPRSC is gratefully acknowledged.
\end{acknowledgement}

\end{document}